\documentclass{article}
\usepackage[utf8]{inputenc}
\usepackage{geometry}
\usepackage{amsthm}
\usepackage{amssymb}
\usepackage{amsmath}
\usepackage{graphicx}
\usepackage{float}
\usepackage{tikz}
\usepackage{ifthen}
\usepackage{mathtools}
\usepackage[mathscr]{euscript}
\newtheorem{theorem}{Theorem}[section]
\newtheorem{lemma}[theorem]{Lemma}
\newtheorem{corollary}[theorem]{Corollary}
\newtheorem{prop}[theorem]{Proposition}

\theoremstyle{definition}
\newtheorem{definition}[theorem]{Definition}

\newcommand{\kemeny}{\mathscr{K}}
\newcommand{\hidden}[1]{}
\newcommand{\rone}{r_1}
\newcommand{\rtwo}{r_2}
\newcommand{\Gnv}{G\setminus\{v\}}

\title{A 1-Separation Formula for the Graph Kemeny Constant and Braess Edges}
\author{
    Nolan Faught\footnote{
        Department of Mathematics, Brigham Young University, Provo UT, USA, faught3@gmail.com
    },
    Mark Kempton\footnote{
        Department of Mathematics, Brigham Young University, Provo UT, USA, mkempton@mathematics.byu.edu
    },
    Adam Knudson\footnote{
        Department of Mathematics, Brigham Young University, Provo UT, USA, adamarstk@yahoo.com
    }
}
\date{}

\begin{document}

\maketitle

\begin{abstract}
    Kemeny's constant of a simple connected graph $G$ is the expected length of a random walk
    from $i$ to any given vertex $j \neq i$. We provide a simple method for computing Kemeny's constant
    for 1-separable via effective resistance methods from electrical network theory. Using this
    formula, we furnish a simple proof that the path graph on $n$ vertices maximizes Kemeny's constant
    for the class of undirected trees on $n$ vertices. Applying this method again, we simplify existing
    expressions for the Kemeny's constant of barbell graphs and demonstrate which barbell maximizes
    Kemeny's constant. This 1-separation identity further allows us to create sufficient conditions for
    the existence of Braess edges in 1-separable graphs. We generalize the notion of the Braess edge to
    Braess sets, collections of non-edges in a graph such that their addition to the base graph increases
    the Kemeny constant. We characterize Braess sets in graphs with any number of twin pendant vertices,
    generalizing work of Kirkland et.~al.~\cite{kirkland2016kemeny} and Ciardo \cite{ciardo2020braess}.
\end{abstract}

\section{Introduction}
    Kemeny's constant is an invariant of a Markov chain that represents the expectation of the
    hitting times. 
    Kemeny's constant of a Markov chain $P$ is computed with the sum
    \begin{equation*}
        \kemeny(P) = \sum_j \pi_j m_{ij},
    \end{equation*}
    where $\pi_j$ is the $j$-th entry of the stationary distribution of the Markov chain and $m_{ij}$
    is the hitting time of $j$ for a random walk with initial state $i$. Although this sum is not
    well-defined for all Markov chains, it is well-defined for random walks on simple connected graphs.
    In this context, Kemeny's constant measures the expected length of a random walk between two
    randomly chosen vertices, and serves as a measure of how well-connected a graph is.
    For a more comprehensive discussion
    of Kemeny's constant, we refer the reader
    to \cite{catral2010kemeny}.
    This has mathematical applications to graph theory, and real-world applications in robotics
    \cite{patel2015robotic}, web navigation \cite{levene2002kemeny}, and mathematical chemistry \cite{li2019multiplicative}.
    
    Intuitively, adding edges to a graph increases connectivity and generates shortcuts for random walks,
    so we would expect Kemeny's constant to decrease when we add connections to a graph. Typically this
    is true, but the authors of \cite{kirkland2016kemeny} demonstrate that most trees contain a pair of
    vertices
    such that adding an edge between them increases Kemeny's constant. This phenomenon is an instance of 
    Braess' traffic-planning paradox, presented in \cite{braess1968paradoxon} (see
    \cite{braess2005paradox} for an English translation), in which the deletion of an edge from a network
    improves some properties of connectivity. A non-edge $e$ of a graph $G$ that results in an increase
    of $\kemeny(G)$ when inserted is a \emph{Braess edge}.
    
    The authors of \cite{kirkland2016kemeny} go on to prove that \emph{twin pendants} (degree one vertices
    adjacent to the same vertex) in trees are always Braess. Building on this, \cite{ciardo2020braess}
    proves that twin pendants are always Braess on any nontrivial connected graph.
    
    
    In the work of \cite{kirkland2016kemeny} and \cite{ciardo2020braess}, the graphs in which Braess edges are observed to occur all have a \emph{1-separation}, that is, a single vertex whose removal disconnects the graph.  This gives rise to the question if Braess edges occur in other graphs with 1-separations.  From this perspective, Kemeny's constant in a graph with a 1-separation is naturally approached using the \emph{effective resistance} (also called \emph{resistance distance}) from electrical network theory.  Kemeny's constant can be expressed using a formula involving effective resistances in a graph (see Lemma \ref{thm:kemeny} below).  In addition, effective resistance is easily computed in graphs with a 1-separation (see Proposition \ref{prop:1sepres} below).  This provides the motivation for the present paper, in which we study Kemeny's constant in graphs with a 1-separation.
    
     In this paper, we give a 1-separation formula for Kemeny's constant, that is, a formula for Kemeny's constant of a graph with a 1-separation given Kemeny's constant of simpler subgraphs.  See Theorem \ref{thm:kem1sepformula} below.  The proof of this Theorem makes use of the effective resistance and its nice behavior in graphs with a 1-separation.  Furthermore, we demonstrate several uses of this 1-separation formula.  First, we give simple expressions
    for Kemeny's constant of barbell graphs, which were studied in \cite{breen2019computing}. Barbell graphs are of interest in the study of Kemeny's constant, as they are believed, based on empirical computation, to maximize Kemeny's constant among graphs on a given number of vertices.  From \cite{breen2019computing}, it is known that certain barbells on $n$ vertices have Kemeny's constant on the order of $n^3$, and that order $n^3$ is the largest Kemeny's constant can be.  Our 1-separation formula allows us to give an exact expression for Kemeny's constant in barbell graphs that is much simpler than that found in \cite{breen2019computing}, and we are able to determine what barbell has the largest Kemeny's constant among all barbell graphs on $n$ vertices.
    Second, we use our formula to prove that, among all trees on $n$ vertices, the path graph has the largest Kemeny's constant. 
    Finally, we will return our focus to Braess' paradox and use our formula to find conditions under which a graph with a 1-separation has a Braess edge or a Braess set of edges.  We generalize work of \cite{kirkland2016kemeny} and \cite{ciardo2020braess} to include any number of twin pendants.
    

    
    \subsection{Notation and Preliminaries}
        We introduce a few definitions and results that we will use. We denote by $r_G(i,j)$ the \emph{effective resistance} between vertex $i$ and $j$, considering the graph as an electric circuit with each edge representing a unit resistor. This quantity is given by $r_G(i,j) = (e_i - e_j)^TL^\dag(e_i-e_j)$ where $e_i$ is the vector with a 1 in the $i$-th position and zeros elsewhere and $L^\dag$ is the Moore-Penrose pseudoinverse of the graph Laplacian matrix (see \cite{bapat2010graphs}).
        \begin{lemma}[Corollary 1 of \cite{palacios2011broder}]\label{thm:kemeny}
            Suppose that $G = (V, E)$ is a simple connected graph, where $R$ denotes the matrix whose
            $(i, j)$-th entry is the effective resistance between $i$ and $j$, $d$ the vector whose $i$-th entry is the degree of vertex $i$
            , and $m = |E|$. Kemeny's constant of the graph is related to the effective resistance
            by the identity
            \begin{equation*}
                \kemeny(G) = \frac{d^T Rd}{4m} = \frac{1}{4m}\sum_{i, j \in G} d_i d_j r_G(i, j).
            \end{equation*}
        \end{lemma}
    
    The notion of the moment is proposed for rooted trees in \cite{ciardo2020kemeny}. We extend this
    concept to the more general class of simple connected graphs.
    \begin{definition}
        Let $G = (V, E)$ be a simple connected graph. Let $e_v$ denote the vector with a 1 in the $v$-th position and zeros elsewhere. The \emph{moment} of $v \in V$ is
        \begin{equation*}
            \mu(G, v) = d^T R e_v = \sum_{i\in V(G)} d_i r_G(i,v).
        \end{equation*}
    \end{definition}
    
    \begin{definition}
        Let $G_1, G_2$ be simple connected graphs and with labelled vertices $v_1 \in V(G_1)$, and
        $v_2 \in V(G_2)$. The \emph{1-sum} $G = G_1\oplus_{v_1,v_2} G_2$ is the graph created by 
        taking a copy of $G_1, G_2$, removing $v_1$, and replacing every edge of the form $\{i, v_1\} \in E(G_1)$
        with $\{i, v_2\}$. We often omit the subscript when the choice and/or labelling of vertices
        is clear. We say $G_1\oplus_vG_2$ has a \emph{1-separation}, and that $v$ is a \emph{1-separator} or \emph{cut vertex.}
    \end{definition}
    
    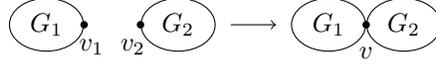
\begin{figure}[H]
        \centering
        \begin{tikzpicture}
    \tikzstyle{every node}=[circle, draw=none, fill=white, minimum width = 6pt, inner sep=1pt]
    \draw[] (1.75, 1)ellipse(14pt and 10pt);
    \draw[] (3.5, 1)ellipse(14pt and 10pt);
    \draw{
    (1.75,1)node[]{$G_1$}
    (3.5,1)node[]{$G_2$}
    (2.25,1)node[fill=black, minimum width = 2pt, label={[shift={(0.1,-.6)}]{$v_1$}}]{}
    (3,1)node[fill=black, minimum width = 2pt, label={[shift={(-0.1,-.6)}]{$v_2$}}]{}
    };
    
    \draw[->] (4.2,1) -- (4.8,1);
    
    \draw[] (5.5, 1)ellipse(14pt and 10pt);
    \draw[] (6.5, 1)ellipse(14pt and 10pt);
    \draw{
    (5.5,1)node[]{$G_1$}
    (6.5,1)node[]{$G_2$}
    (6,1)node[fill=black, minimum width = 2pt, label={[shift={(0,-.6)}]{$v$}}]{}
    };
    \end{tikzpicture}
        \caption{The graph $G=G_1\oplus_vG_2$ created from $G_1$ and $G_2$}
        \label{fig:vertexsum}
    \end{figure}

    \begin{prop}[Theorem 2.5 of \cite{barrett2019resistance}, Cut Vertex Theorem]\label{prop:1sepres}
        Let $G$ be the 1-sum of $G_1, G_2$ with labelled vertex $v$. For $i \in V(G_1)$, $j \in V(G_2)$,
        \begin{equation*}
            r_G(i,j) = r_{G_1}(i,v)+r_{G_2}(v,j).
        \end{equation*}
    \end{prop}

\section{The 1-Separation Formula}\label{sec:1sep}
    In this section we will derive various useful expressions for Kemeny's constant of graphs with 1-separations. These will make it easier to compute Kemeny's constant for 1-connected graphs, lead to a result about trees with maximal Kemeny's constant, and help us determine a sufficient condition on graph structures that may allow for Braess edges.
    
     Theorem \ref{thm:kem1sepformula} below is a special case of Theorem \ref{thm:kemenymany1seps} and used as a base cases in its proof. However, it is useful enough in its own right that we have included it as a theorem of its own.
    
    
    \begin{theorem}\label{thm:kem1sepformula}
    Let $G$ be a graph with a 1-separator $v$. Let $G_1, G_2$ be the two graphs of the 1-separation so $G=G_1\oplus_vG_2$ and $m = |E(G)| = |E(G_1)| + |E(G_2)| = m_1 + m_2.$ Then we have
    \begin{align*}
        \kemeny(G) =&\: \frac{m_1\left(\kemeny(G_1) + \mu(G_2, v)\right) + m_2\left(\kemeny(G_2) + \mu(G_1, v)\right)}{m_1+m_2}.
    \end{align*}
    \end{theorem}
    \begin{proof}
        We simplify the expression for Kemeny's constant in Lemma 1.1 via application of Proposition \ref{prop:1sepres}.
        Throughout the proof, we denote $r_{G_1}$ and $r_{G_2}$ with the shorthand $r_1$ and $r_2$ and $d_{v_1}$
        refers to the degree of $v$ as a vertex of $G_1$, while $d_{v_2}$ its degree as a vertex of $G_2$.
        Partition the vertex set $V$ into $V_1 = V(G_1) \setminus \{v\}$, $V_2 = V(G_2) \setminus v$, and $\{v\}$,
        so that we now consider the sum over each part. Due to the symmetricity of effective resistance, each
        case $i \in V_k$, $j \in V_\ell$, $k \neq \ell$ is similar to $i \in V_\ell$, $j \in V_k$.
        
        Let $i \in V_1$ and $j \in V_2$, then the sum over the parts is
        \begin{align*}
            \frac{1}{4m} \sum_{i \in V_1, j \in V_2} d_i d_j r_G(i, j)
            &= \frac{1}{4m} \sum_{i \in V_1} d_i \sum_{j \in V_2} d_j \left(r_1(i, v) + r_2(v, j)\right) \\
            &= \frac{1}{4m} \sum_{i \in V_1}d_i \left[
                \sum_{j \in V_2} d_j r_2(j, v) + \sum_{j \in V_2} d_j r_1(i, v)
            \right] \\
            &= \frac{1}{4m} \left(
                \sum_{i \in V_1} d_i \mu(G_2, v) + \sum_{j \in V_2} d_j \left[
                    \sum_{i \in V_1} d_i r_1(i, v)
                \right]
            \right) \\
            &= \frac{1}{4m} ((2m_1-d_{v_1}) \mu(G_2, v) + (2m_2-d_{v_2}) \mu(G_1, v)) \\
            &= \frac{m_1 \mu(G_2, v) + m_2 \mu(G_1, v)}{2m}-\frac{d_{v_1}\mu(G_2,v)+d_{v_2}\mu(G_1,v)}{4m}.
        \end{align*}
        Taking the sum over $i, j \in V_1$,
        \begin{align*}
            \frac{1}{4m} \sum_{i, j \in V_1} d_i d_j r_G(i, j)
            &= \frac{1}{4m} \sum_{i, j \in V(G_1)} d_i d_j r_1(i, j)
             - \frac{2}{4m} \sum_{i \in V_1} d_{v_1} d_i r_1(i, v) \\
            &= \frac{m_1 \kemeny(G_1)}{m} - \frac{d_{v_1}  \mu(G_1, v)}{2m}.
        \end{align*}
        The case for $i, j \in V_2$ is similar, with the sum evaluating to
        \begin{equation*}
            \frac{1}{4m} \sum_{i, j \in V_2} d_i d_j r_G(i, j)
            = \frac{m_2 \kemeny(G_2)}{m} - \frac{d_{v_2}\mu(G_2, v)}{2m}.
        \end{equation*}
        Finally, the sum over $i \in V_1$ with $j = v$ is
        \begin{align*}
            \frac{1}{4m} \sum_{i \in V_1} d_i d_v r_G(i, v)
            &= \frac{d_v}{4m} \sum_{i \in V_1} d_i r_1(i, v) = \frac{(d_{v_1}+  d_{v_2}) \mu(G_1, v)}{4m},
        \end{align*}
        which is similar to the sum over $i \in V_2$ with $j = v$,
        \begin{equation*}
            \frac{1}{4m} \sum_{i \in V_1} d_i d_v r_2(i, v)
            = \frac{(d_{v_1} + d_{v_2}) \mu(G_2, v)}{4m}.
        \end{equation*}
        
        Combining and doubling the appropriate terms,
        \begin{align*}
            \frac{1}{4m} \sum_{i, j} d_i d_j r_G(i, j)
            =&\: 2 \left(
                \frac{m_1 \mu(G_2, v) + m_2 \mu(G_1, v)}{2m}
           -\frac{d_{v_1}\mu(G_2,v)+d_{v_2}\mu(G_1,v)}{4m} \right)\\& + \frac{m_1 \kemeny(G_1)}{m} - \frac{d_{v_1}\mu(G_1, v)}{2m} 
            + \frac{m_2 \kemeny(G_2)}{m} - \frac{ d_{v_2} \mu(G_2, v)}{2m}\\ &+ 2 \left(
                \frac{(d_{v_1} + d_{v_2}) \mu(G_1, v)}{4m}
            \right) + 2 \left(
                \frac{(d_{v_1} + d_{v_2}) \mu(G_2, v)}{4m}
            \right) \\
            =&\: \frac{m_1(\kemeny(G_1) + \mu(G_2, v)) + m_2(\kemeny(G_2) + \mu(G_1, v))}{m}.
        \end{align*}\end{proof}
    
    We will also consider graphs which have multiple 1-separations and get an expression for Kemeny's constant for 
    such graphs. Before that we prove a useful Lemma.
    
    \begin{lemma}\label{thm:momentmany1seps}
    Let $G = G_1\oplus_{v_{1,2}}G_2\oplus_{v_{2,3}}\hdots\oplus_{v_{n-1,n}}G_n$. Then for some $v_{0,1} = v_0\in G_1$ we have
    \begin{equation*}
        \mu(G,v_0) = \sum_{i=1}^n\mu(G_i,v_{i-1,i}) + 2\sum_{i=2}^nr(v_{i-2,i-1},v_{i-1,i})\sum_{j=i}^nm_j.
    \end{equation*}
    \end{lemma}
    \begin{proof}
    We proceed by induction on $n$.  For $n=1$ this is immediate.
    
    Suppose $n=2$. In this case let $d_G(i)$ denote the degree of vertex $i$ in $V(G).$ Then we have

    \begin{align*}
        \mu(G,v_0) =& \sum_{i\in G}d_G(i)r(i,v_0)\\
        =&\sum_{\substack{i\in G_1\\i\neq v_{1,2}}}d_{G_1}(i)r_1(i,v_0) + (d_{G_1}(v_{1,2}) + d_{G_2}(v_{1,2}))r(v_{1,2},v_0)\\
        &+ \sum_{\substack{i\in G_2\\i\neq v_{1,2}}}d_{G_2}(i)(r_2(i,v_{1,2})+r(v_{1,2},v_0))\\
        =&\:\mu(G_1,v_0)+\mu(G_2,v_{1,2})+ 2m_2r(v_{1,2}, v_0).
    \end{align*}
    Now suppose for some $k\geq2$ the expression is true. Let $G = G_1\oplus_{v_{1,2}}\cdots \oplus_{v_{k-1,k}}\left(G_k\oplus_{v_{k,k+1}}G_{k+1}\right)$. Note that $|E\left(G_k\oplus_{v_{k,k+1}}G_{k+1}\right)| = m_k+m_{k+1}.$ By the inductive hypothesis we have 
    \begin{align*}
        \mu(G,v_0) =&\sum_{i=1}^{k-1}\mu(G_i,v_{i-1,i}) + 2\sum_{i=2}^{k-1}r(v_{i-2,i-1},v_{i-1,i})\sum_{j=i}^{k-1}m_j + \mu\left(\left(G_k\oplus_{v_{k,k+1}}G_{k+1}\right),v_{k-1,k}\right)\\
        &+2\sum_{i=2}^{k}r(v_{i-2,i-1},v_{i-1,i})(m_k+m_{k+1})
    \end{align*}
    But $\mu\left(\left(G_k\oplus_{v_{k,k+1}}G_{k+1}\right),v_{k-1,k}\right) = \mu(G_k,v_{k-1,k}) + \mu(G_{k+1},v_{k,k+1}) + 2m_{k+1}r(v_{k-1,k}, v_{k,k+1})$.
    Substituting this into the above expression we will get the result.
    \end{proof}
    
    
    \hidden{
    \begin{theorem}\label{thm:kemenytwo1seps}
    Let $G=G_1\oplus_{v_{1,2}}G_2\oplus_{v_{2,3}}G_3$. Then we have
    \begin{align*}
        \kemeny(G) =&\: \frac{m_1[\kemeny(G_1)+\mu(G_2, v_{1,2}) + \mu(G_3, v_{2,3})] + m_2[\kemeny(G_2) + \mu(G_1, v_{1,2}) + \mu(G_3, v_{2,3})]}{m_1+m_2+m_3}\\
        &+\frac{m_3[\kemeny(G_3) + \mu(G_1, v_{1,2})+\mu(G_2, v_{2,3})]+2m_1m_3r(v_{1,2}, v_{2,3})}{m_1+m_2+m_3}.
    \end{align*}
    \end{theorem}
    \begin{proof}
    By Theorem \ref{thm:kem1sepformula} we have the following.
    \begin{align*}
        \kemeny(G) =&\:\frac{m_1[\kemeny(G_1) + \mu((G_2\oplus G_3), v_{1,2})] + (m_2+m_3)[\kemeny(G_2\oplus G_3) +  \mu(G_1, v_{1,2})]}{m_1+m_2+m_3}
    \end{align*}
    Using Theorem \ref{thm:kem1sepformula} again and Theorem \ref{thm:momentmany1seps} we get
    \[\kemeny(G_2\oplus G_3) = \frac{m_2[\kemeny(G_2) + \mu(G_3, v_{2,3})] + m_3[\kemeny(G_3) + \mu(G_2, v_{2,3})]}{m_2+m_3}\]
    \[\mu((G_2\oplus G_3), v_{1,2}) = \mu(G_2, v_{1,2}) + \mu(G_3, v_{2,3}) + 2m_3r(v_{1,2},v_{2,3}).\]
    Substituting in these expressions yields the result.
    \end{proof}
    }
    
    Now we are ready to find Kemeny's constant for a graph with multiple 1-separations.
    \begin{theorem}\label{thm:kemenymany1seps}
    Let $G = G_1\oplus_{v_{1,2}}G_2\oplus_{v_{2,3}}\cdots\oplus_{v_{n-1,n}}G_n$. Let $q_{i,j}$ be as follows.
    \begin{align*}
        q_{i,j} =& \begin{cases} (j-1,j) &\text{ $j>i$}\\
        (j,j+1) &\text{ $j < i$}\end{cases}
    \end{align*}
    Then Kemeny's constant of $G$ is given by the following.
    \begin{equation*}
        \kemeny(G) = \frac{\sum\limits_{i=1}^{n}m_i\left(\kemeny(G_i)+\sum\limits_{j\neq i}\mu(G_j, v_{q_{i,j}})\right)+2\sum\limits_{\substack{1\leq i < j\leq n\\ j-i\geq2}}m_im_jr(v_{i,i+1},v_{j-1,j})}{\sum\limits_{i=1}^nm_i}.
   \end{equation*}
    \end{theorem}
    \begin{proof}
    For $n=1$ this is true. For $n=2$ this is true by Theorem \ref{thm:kem1sepformula}. 
    
    
    Suppose this expression is true for some $n=k\geq2$. Let 
    \[G = G_1 \oplus_{v_{1,2}}G_2 \oplus_{v_{2,3}}\hdots\oplus_{v_{k-1,k}}\left(G_k\oplus_{v_{k,k+1}}G_{k+1}\right).\]
    Let $H_i = G_i$ for $i < k$ and let $H_k = (G_k\oplus_{v_{k,k+1}}G_{k+1})$. Then we can express $G$ as
    \[G = H_1 \oplus_{v_{1,2}}H_2 \oplus_{v_{2,3}}\hdots\oplus_{v_{k-1,k}}H_k.\]
    If $m_i = |E(G_i)|$ and $\tilde{m}_i = |E(H_i)|$ then $m_i = \tilde{m}_i$ for $i < k$ and $\tilde{m}_k = m_k + m_{k+1}.$ Notice by Theorem \ref{thm:kem1sepformula} and Theorem \ref{thm:momentmany1seps} we have the following.
    \begin{equation}\label{eq:keminductive}
    \kemeny(H_k) = \frac{m_k\left(\kemeny(G_k) + \mu(G_{k+1},v_{k,k+1})\right) + m_{k+1}\left(\kemeny(G_{k+1}) + \mu(G_k, v_{k,k+1})\right)}{m_k+m_{k+1}}
    \end{equation}
    
    \begin{equation}\label{eq:muinductive}
    \mu(H_k, v_{k-1,k}) = \mu(G_k, v_{k-1,k}) + \mu(G_{k+1}, v_{k, k+1}) + 2m_{k+1}r(v_{k-1,k}, v_{k,k+1}).
    \end{equation}
    
    By the inductive hypothesis we have the following.
    \begin{equation}\label{eq:inductivetotalkem}
        \kemeny(G) = \frac{\sum_{i=1}^k\tilde{m}_i\left(\kemeny(H_i)+\sum_{j\neq i}\mu(H_j, v_{q_{i,j}})\right)+2\sum_{\substack{1\leq i < j\leq k\\j-i\geq 2}}\tilde{m}_i\tilde{m}_jr(v_{i,i+1},v_{j-1,j})}{\sum_{i=1}^k\tilde{m}_i}
    \end{equation}
    
    Now we will examine the numerator terms of (\ref{eq:inductivetotalkem}). Using (\ref{eq:muinductive}), for $i < k$ a term in the first term in the numerator looks like the following.
    \begin{equation}\label{eq:kemnumsumi_less_k}
        m_i\left(\kemeny(G_i) + \sum_{\substack{j\neq i\\j<k}}\mu(G_j, v_{q_{i,j}}) + \mu(G_k, v_{k-1,k}) + \mu(G_{k+1}, v_{k,k+1}) + 2m_{k+1}r(v_{k-1,k}, v_{k,k+1})\right)
    \end{equation}
    This gives the $i$-th term in the first sum of (\ref{eq:inductivetotalkem}) and a resistance term.
    
    Now, for $i=k$, (\ref{eq:keminductive}) gives that a term in the first sum of (\ref{eq:inductivetotalkem}) looks like the following.
    \begin{equation}\label{eq:kemnumsumk_is_i}
        m_k\left(\kemeny(G_k) + \sum_{i\neq k}\mu(G_i, v_{q_{i,k}})\right) + m_{k+1}\left(\kemeny(G_{k+1}) + \sum_{i\neq k+1}\mu(G_i, v_{q_{i,k}})\right)
    \end{equation}
    
    Now consider the second term of (\ref{eq:inductivetotalkem}). Terms in that sum that involve $H_k$ will look like the following.
    \begin{equation}\label{eq:incompleteresistance}
        2m_im_kr(v_{i,i+1}, v_{k-1,k}) + 2m_im_{k+1}r(v_{i,i+1}, v_{k-1,k})
    \end{equation}
    
    But by combining (\ref{eq:incompleteresistance}) with the resistance term from (\ref{eq:kemnumsumi_less_k}) and using Proposition \ref{prop:1sepres} we get
    \begin{align}
        &2m_im_kr(v_{i,i+1}, v_{k-1,k}) + 2m_im_{k+1}r(v_{i,i+1}, v_{k-1,k}) + 2m_im_{k+1}r(v_{k-1,k}, v_{k,k+1})\nonumber\\
    =&\: 2m_im_kr(v_{i,i+1}, v_{k-1,k}) + 2m_im_{k+1}r(v_{i,i+1}, v_{k,k+1}).\label{eq:kemresistancebit}
    \end{align}
    Thus combining (\ref{eq:kemnumsumi_less_k}), (\ref{eq:kemnumsumk_is_i}), and (\ref{eq:kemresistancebit})
    we get that
    \[\kemeny(G) = \frac{\sum_{i=1}^{k+1}m_i\left(\kemeny(G_i)+\sum_{j\neq i}\mu(G_j, v_{q_{i,j}})\right)+2\sum_{\substack{1\leq i < j\leq k+1\\j-i\geq 2}}m_im_jr(v_{i,i+1},v_{j-1,j})}{\sum_{i=1}^{k+1}m_i}\]
    and the result is proven.
    
    \end{proof}
    
    Using the previous Theorem to consider a graph with multiple components all 1-summed at the same vertex gives the following result.
    \begin{corollary}\label{cor:kemenymanyG1v}
    Let $G = G_1\oplus_vG_2\oplus_v\hdots\oplus_vG_n$ and $m_i = |E(G_i)|$. Then Kemeny's constant of G is
    \begin{equation*}
        \kemeny(G) = \frac{\sum\limits_{i=1}^nm_i\left(\kemeny(G_i) + \sum\limits_{j\neq i}\mu(G_j,v)\right)}{\sum\limits_{i=1}^nm_i}.
    \end{equation*}
    \end{corollary}
    \begin{proof}
    If $v_{i,i+1} = v_{j-1,j} = v$ then $r(v_{i,i+1}, v_{j-1,j}) = 0$. The result then follows directly from Theorem \ref{thm:kemenymany1seps}.
    \end{proof}

\section{Applications of the 1-Separation Formula}
In this section we will demonstrate how these results can simplify the computation of Kemeny's constant for graphs with a 1-separation. We will obtain an expression for Kemeny's constant of barbell graphs. We also provide a result about trees with maximal Kemeny's constant. Now, we state without proof the resistance,  Kemeny's constant, and moment of some graphs that are both easy to compute and will prove useful.  See \cite{bapat2010graphs} for details about how to compute effective resistance.

\begin{prop}\label{prop:someresistances}
Let $K_n$ and $P_n$ be, respectively, the complete graph and path graph on $n$ vertices $\{1, 2, \hdots, n\}$. Then 
\begin{gather*}
     r_{K_n}(i,j) =\frac2n\\
     r_{P_n}(i,j) = d(i,j) 
\end{gather*}
where $d(i,j)$ is the distance from $i$ to $j$.
\end{prop}

\begin{prop}\label{prop:somekemenys}
Let $K_n$, $P_n$, and $S_n$ be, respectively, the complete graph, path graph, and star graph on $n$ vertices $\{1, 2, \hdots, n\}$. Let 1 be the central vertex of $S_n$. Then 
\begin{align*}
    \kemeny(K_n) =& \frac{(n-1)^2}{n} & \kemeny(P_n) =& \frac{2n^2-4n+3}{6} & \kemeny(S_n) =&\:n-\frac32\\
    \mu(K_n, j) =& \frac{2(n-1)^2}{n} & \mu(P_n, j) =& (n-j)^2 + (j-1)^2 & \mu(S_n,1) =&\:n-1.
\end{align*}
\end{prop}

\subsection{Barbell Graphs}\label{sec:barbell}
    In \cite{breen2019computing} barbell graphs are studied because of their large Kemeny's constants by 
    using tools of spectral graph theory. They give the following definition for these barbell graphs.
    \begin{definition} (Definition 1.3 of \cite{breen2019computing})
    The graph $B(k,a,b,c)$ on $ka + b + c$ vertices is formed by taking $k$ copies of $P_a$ (path on $a$ vertices) and putting a clique at both ends to ``glue'' the paths together; we then connect all vertices of a $K_b$ to one set of neighbors in the graph with degree $k$ and a $K_c$ to the other set of neighbors in the graph with degree $k$.
    \end{definition}
    
    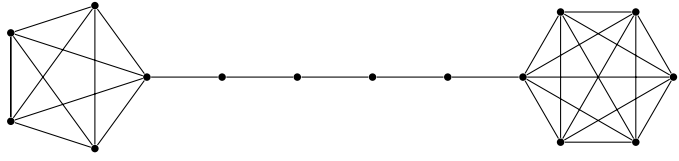
\begin{figure}[H]
    \centering
    \begin{tikzpicture}
        \tikzstyle{every node}=[circle, fill=black, inner sep=1pt]
           \node (1) at (1,0) {};
           \node (2) at (0.309,0.951) {};
           \node (3) at (-0.809,0.588) {};
           \node (4) at (-0.809,-0.588) {};
           \node (5) at (0.309,-0.951) {};
           
        \node (6) at (2, 0) {};
        \node (7) at (3,0) {};
        \node (8) at (4,0) {};
        \node (9) at (5,0) {};
        \node (10) at (6,0) {};

        \node (11) at (6.5, 0.866) {};
        \node (12) at (7.5, 0.866) {};
        \node (13) at (8,0) {};
        \node (14) at (7.5, -0.866) {};
        \node (15) at (6.5, -0.866) {};
       \draw{
          (1)--(2)--(3)--(4)--(5)--(1)
          (1)--(3)--(4)--(1)
          (2)--(4)
          (3)--(5)--(2)
          (1)--(6)--(7)--(8)--(9)--(10)
          (10)--(11)--(12)--(13)--(14)--(15)--(10)
          (10)--(12)--(14)--(10)--(13)--(15)--(11)--(13)
          (11)--(14)
          (12)--(15)
       };
    \end{tikzpicture}
    \caption{The graph $B(1,6,4,5)$}
    \label{fig:barbell}
\end{figure}

    We will concern ourselves only with barbells with $k=1$. Using Theorem \ref{thm:kemenymany1seps} makes it much easier to get an expression for Kemeny's constant of 1-connected barbell graphs, and the resulting expression is much simpler that that found in \cite{breen2019computing}.
    
    \hidden{
    \begin{lemma}
        Let $K_n$ be a complete graph on $n$ vertices with labels from $1$ to $n$, then
        \begin{equation*}
            r_{K_n}(i, j) = \begin{cases}
                \frac{2}{n} &\text{ if } i \neq j \\
                0 &\text{otherwise}
            \end{cases}
        \end{equation*}
    \end{lemma}
    }
    
    \begin{theorem}\label{thm:barbell}
        Kemeny's Constant of a Barbell graph $G = B(1, a, b, c)$ is given by
        \begin{align*}
        \kemeny(G) &= \frac{1}{m}\left[\binom{b+1}{2}\left(\frac{b^2}{b+1}+(a-1)^2+\frac{2c^2}{c+1}\right)\right.+ (a-1)\left(\frac{2a^2-4a+3}{6}+\frac{2b^2}{b+1}+\frac{2c^2}{c+1}\right)\\
        &\left.\quad\quad\quad+\binom{c+1}{2}\left(\frac{c^2}{c+1}+(a-1)^2+\frac{2b^2}{b+1}\right)+2\binom{b+1}{2}\binom{c+1}{2}(a-1)\right]
        \end{align*}
        \hidden{
        \begin{align*}
            \kemeny(G)
            =&\: \frac 1m \left[
                \frac{(a-1)^3}{3} + (a-1)^2 \left[
                    \binom{b+1}{2} + \binom{c+1}{2}
                \right]
            \right. \\
            &\hspace{2em} \left.
                +\: 2(a-1) \left[
                    \binom{b+1}{2} \binom{c+1}{2} + \frac{b^2}{b+1} + \frac{c^2}{c+1} + \frac{1}{12}
                \right]
            \right. \\
            &\hspace{2em} \left.
                +\: \binom{b+1}{2} \left(
                    \frac{b^2}{b+1} + \frac{2c^2}{c+1}
                \right)
                +\binom{c+1}{2} \left(
                    \frac{2b^2}{b+1} + \frac{c^2}{c+1}
                \right)
            \right]
        \end{align*}
        }
        where $m=\binom{b+1}{2}+\binom{c+1}{2} + a-1$.
    \end{theorem}
    \begin{proof}
        The graph $B(1, a, b, c)$ is the one sum of $K_{b+1}, K_{c+1}$, and $P_a$. Notice
        that $r_{P_a}(1,a) = a-1.$ With this fact and with Proposition \ref{prop:somekemenys}
        the result is an immediate consequence of Theorem \ref{thm:kemenymany1seps} with $n=3$.
    \end{proof}
    As shown in \cite{breen2019computing}, the barbell with maximum Kemeny constant (among all barbells on $n$ vertices) will occur when $a,b,c$ are all close to $\frac n3$. A more careful analysis of the expression in Theorem \ref{thm:barbell} shows that $B(1, \frac{n}{3}+2, \frac{n}{3}-1, \frac{n}{3}-1)$ will be the actual barbell with largest Kemeny constant. In the next two corollaries one can see that Kemeny's constant of these two barbells have the same order of magnitude and differ only in the lower terms.

    \begin{corollary}\label{cor:bigbarbell}
        Kemeny's Constant of a Barbell graph $B(1, \frac n3, \frac n3, \frac n3)$ is given by
        \begin{equation*}
            \kemeny\left(B\left(1, \frac n3, \frac n3, \frac n3\right)\right) = \frac{1}{54}\left[n^3+3n^2+24n-36+\frac{-513n^2+1782n-1701}{n^3+9n^2+9n-27}\right].
        \end{equation*}
    \end{corollary}
    
    \begin{corollary}\label{cor:biggestbarbell}
        Kemeny's Constant of a Barbell graph $B(1,\frac{n}{3}+2, \frac{n}{3}-1, \frac{n}{3}-1)$ is given by
        \begin{equation*}
            \kemeny\left(B\left(1,\frac{n}{3}+2, \frac{n}{3}-1, \frac{n}{3}-1\right)\right)=\:\frac{1}{54}\left[n^3+3n^2+60n-270+\frac{297n^2-729n+5832}{n^3+9n}\right].
        \end{equation*}
    \end{corollary}

\subsection{Trees}
In this section we use Theorem \ref{thm:kem1sepformula} to show that the path graph has the largest Kemeny's constant among trees of order $n$. We first show that the path graph also has the largest moment among trees of order $n$. While this was shown by Proposition 5.2 of \cite{ciardo2020kemeny}, for the sake of completeness we give a different proof using results from this paper. We note that in trees the effective resistance between vertices $i,j$ is the graph distance between $i,j$ (see \cite{bapat2010graphs}). 

\begin{lemma}\label{lem:pathmaxmoment}
Let $T_n$ denote a tree on $n$ vertices and let $v \in V(T_n).$ Then $\mu(T_n,v) \leq \mu(P_n, 1)$.
\end{lemma}
\begin{proof}
This is trivial to check for $n=1,2,3$.

Suppose for some $k \geq3$ that $\mu(T_k, v) \leq \mu(P_k,1).$ Let $l \leq k$, $n = k+l-1$, and $T_n = T_k \oplus_w T_l$. Then without loss of generality, by Theorem \ref{thm:momentmany1seps} we have
\begin{align*}
    \mu(T_n,v) =&\:\mu(T_k,v) + \mu(T_l, w) + 2(l-1)r(w,v)\\
    \leq&\:\mu(P_k,1) + \mu(P_l,1) + 2(l-1)(k-1)\\
    =&\:\mu(P_n,1).
\end{align*}
Thus for all $n$, $\mu(T_n, v) \leq \mu(P_n, 1).$
\end{proof}

    \begin{theorem}\label{thm:pathmax}
    Among all trees on $n$ vertices, the path $P_n$ maximizes Kemeny's constant.
    \end{theorem}
    \begin{proof}
    Let $T_n$ be a tree of order $n$. The result is trivial for $n=1, 2, 3$. 
    
    
    We proceed with the inductive step. Suppose for some $k \geq 3$ the path graph $P_k$ has maximal Kemeny's constant among trees of order $k$. Let $n = k + l - 1$ for some $l \leq k$. Consider the tree $T_n$ obtained by 1-summing $T_k$ and $T_l$ together at some vertex $v$. Then by Theorem \ref{thm:kem1sepformula} and Lemma \ref{lem:pathmaxmoment} we have
    \begin{align*}
        \kemeny(T_n) =&\: \frac{(k-1)[\kemeny(T_k)+\mu(T_l, v)] + (l-1)[\kemeny(T_l) + \mu(T_k, v)]}{k+l-2}\\
        \leq&\: \frac{(k-1)[\kemeny(P_k)+\mu(P_l, 1)] + (l-1)[\kemeny(P_l) + \mu(P_k, 1)]}{k+l-2}\\
        =&\:\kemeny(P_{k+l-1})\\
        =&\:\kemeny(P_n).
    \end{align*}
    
    Thus among all trees of order $n$, the path graph has maximal Kemeny's constant. 
    \end{proof}
    
\section{Braess Edges}
In this section we will consider the notion of Braess edges introduced in \cite{kirkland2016kemeny}.
We expand the idea of a Braess edge to a Braess set.
We first provide an expression for the difference in Kemeny's constant after edges are added in one component of a graph with a 1-separation. Then, we provide a sufficient condition for such edges to be Braess. Finally, we consider a particular family of graphs as an example.

\begin{definition}[Braess Edges]
\:
\begin{itemize}
    \item A \emph{Braess edge} is a non-edge $e$ of $G$ such that when $e$ is added to $G$, Kemeny's constant increases.
    \item A \emph{Braess set} is a set of non-edges of $G$ such that when the set is added to $G$, Kemeny's constant increases. 
\end{itemize}
\end{definition}


Let $G = G_1\oplus_v G_2$ and $m_1$ and
    $m_2$ denote the number of edges in $G_1$ and $G_2$, respectively. We let $\mathscr{L}$ denote a set of pairs $\{u,v\}$ of $G_2$ that are not edges in $G$ and let $\widehat G = (V, \widehat E)$ denote the
    graph given by adding $\mathscr L$ to $G$ and note $|\widehat E| = m+l$. Note that the choice of $G_2$ is arbitrary, and the following
    result applies to a set of edges $\mathscr L$ in either part of the separation.
    
    \begin{theorem}\label{thm:kemdiff}
    Let $G$ and $\widehat G$ be as defined above. Then we have the following.
        \begin{align*}
            \kemeny(\widehat G) - \kemeny(G) &=\frac{lm_1(\mu(G_1, v) - \kemeny(G_1))}{m(m+l)} + \frac{Am_1^2 + ((A+C)m_2 + Bl)m_1 + C(m_2^2 + lm_2)}{m(m+l)}
        \end{align*}
    Where $A = \mu(\widehat{G}_2,v) - \mu(G_2,v)$, $B = \kemeny(\widehat{G}_2) - \mu(G_2, v)$, and $C = \kemeny(\widehat{G}_2) - \kemeny(G_2).$    
    \end{theorem}
    \begin{proof}
    Theorem \ref{thm:kem1sepformula} gives $\kemeny(G)$ and $\kemeny(\widehat G)$ so all we must do is subtract one from the other. For convenience let $\mu(G) = \mu(G,v)$.
    \begin{align*}
    \kemeny(\widehat G) - \kemeny(G) =&\: \frac{m_1[\kemeny(G_1) + \mu(\widehat{G}_2)] + (m_2+l)[\kemeny(\widehat{G}_2)+\mu(G_1)]}{m+l}\\
    &- \frac{m_1\left[\kemeny(G_1) + \mu(G_2)\right] + m_2\left[\kemeny(G_2) + \mu(G_1)\right]}{m}\\
    =&\:\frac{m\left[m_1[\kemeny(G_1) + \mu(\widehat{G}_2)] + (m_2+l)[\kemeny(\widehat{G}_2+\mu(G_1)]\right]}{m(m+l)}\\
    &-\frac{(m+l)\left[m_1\left[\kemeny(G_1) + \mu(G_2)\right] + m_2\left[\kemeny(G_2) + \mu(G_1)\right]\right]}{m(m+l)}\\
    =&\:\frac{lm_1(\mu(G_1)-\kemeny(G_1))}{m(m+l)}\\
    &+\frac{m(m_2+l)\kemeny(\widehat{G}_2) - (m+l)m_2\kemeny(G_2) + mm_1\mu(\widehat{G}_2) - (m+l)m_1\mu(G_2)}{m(m+l)}\\
    =&\:\frac{lm_1(\mu(G_1)-\kemeny(G_1))}{m(m+l)}\\
    &+\frac{mm_1(\mu(\widehat{G}_2)-\mu(G_2)) + lm_1(\kemeny(\widehat{G}_2)-\mu(G_2))}{m(m+l)}\\
    &+ \frac{mm_2(\kemeny(\widehat{G}_2)-\kemeny(G_2)) + lm_2(\kemeny(\widehat{G}_2) - \kemeny(G_2))}{m(m+l)}\\
    =&\:\frac{lm_1(\mu(G_1, v) - \kemeny(G_1))}{m(m+l)} + \frac{Am_1^2 + ((A+C)m_2 + Bl)m_1 + C(m_2^2 + lm_2)}{m(m+l)}.
    \end{align*}
    Thus we have arrived at the result, taking $A = \mu(\widehat{G}_2,v) - \mu(G_2,v)$, $B = \kemeny(\widehat{G}_2) - \mu(G_2, v)$, and $C = \kemeny(\widehat{G}_2) - \kemeny(G_2).$ 
    \end{proof}
    \begin{corollary}\label{cor:braessset}
    If $Am_1^2 + ((A+C)m_2 + Bl)m_1 + C(m_2^2 + lm_2) > 0$ then $\mathscr{L}$, the set of edges added, is a Braess set.
    \end{corollary}
    \begin{proof}
    It was shown in Proposition 3.2 of \cite{ciardo2020braess} that the first term in Theorem \ref{thm:kemdiff} is non negative, so if the second term is positive then $\mathscr{L}$ is a Braess set of edges. 
    \end{proof}

    Notably, the Braess set is not closed under union nor intersection; for example, the path on $7$ vertices
    with the natural labelling $\{1, 2, \ldots, 7\}$ admits the Braess set $\{(1, 3), (5, 7)\}$, but
    neither of the singleton sets $\{(1, 3)\}$ nor $\{(5, 7)\}$ are Braess themselves. Braess sets are an
    isolated occurrence, it seems.
    
    It is also worth noting that if $A = \mu(\widehat{G}_2) - \mu(G_2) > 0$ then for a sufficiently large $m_1$, the expression in Corollary \ref{cor:braessset} will be positive. That is, if the set of edges added to $G_2$ cause an increase in the moment of $v\in G_2$, then that is a Braess set provided $G_1$ has enough edges. Thus it appears that in graphs with a 1-separation, the moment of the individual graph components plays a more significant role in the existance of Braess sets than Kemeny's constant of the individual components.
    
    We began studying Kemeny's constant in 1-connected graphs because intuition lead to believe that the presence of 1-separations caused Braess edges. This next result seems to support that intuition. 
    \begin{corollary}\label{cor:allbraess}
        Set $G_1 = K_n\oplus P_n$. Then for any graph $G_2$ and any non-edge $e$ of $G_2$, $e$ is Braess in $G = G_1\oplus G_2$ for sufficiently large $n$. In particular, for any $G_2$, there is a $G_1$ such that any pair in $G_2$ is Braess.
    \end{corollary}
    \begin{proof}
    
    Let $G_1 = K_n\oplus P_n$ be the 1-sum of a complete graph to the end of a path graph. Using Lemma \ref{thm:momentmany1seps}, Theorem \ref{thm:kem1sepformula}, and Proposition \ref{prop:somekemenys} we can show that
    \[\kemeny(G_1) = \frac{3n^4-n^3+5n^2-18n+12}{3n(n+2)}\]
    \[\mu(G_1,v) = (n-1)^2\left(n+1+\frac{2}{n}\right)\]
    \[\mu(G_1,v) - \kemeny(G_1) = \frac{3n^4-2n^2-8n+6}{3(n+2)}\]
    where $v$ is the degree 1 vertex of $G_1$. Notice that $|E(G_1)| = \binom{n}{2}+n-1$.
    
    Now fix $G_2$ and choose an edge to add. Note that in terms of Corollary \ref{cor:braessset}, $A, B, C, l$ are all constants. Now plugging all this information into Theorem \ref{thm:kemdiff} and looking only at the numerator we get the following.
    \begin{align*}
        m(m+1)\left(\kemeny(\hat G) - \kemeny(G)\right) =&\:\frac{(n+2)(n-1)}{2}\left(\frac{3n^4-2n^2-8n+6}{3(n+2)}\right)+\frac{A(n+2)^2(n-1)^2}{4}\\
        &+\left((A+C)m_2+B\right)\frac{(n+2)(n-1)}{2}+C(m_2^2+m_2)
    \end{align*}
    Notice the first term is $O(n^5)$ and positive, while the other terms are in $O(n^4)$. Therefore, even if all of $A,B,C < 0$, we would still see an increase in Kemeny's constant, provided $n$ is large enough.
    
    Thus for any $G_2$ and any non-edge $e$ in $G_2$, there exists a $G_1$ such that if $G=G_1\oplus G_2$, $e$ is a Braess edge.
    \end{proof}
    We do not claim that $G_1$ as described in the above proof is the only graph to have this property, but it was used to show that there always exists a graph that has this property. Depending on choice of $G_2$ and the edge added one could potentially have $G_1$ be any graph at all, as was shown in \cite{ciardo2020braess} where $G_2 = P_3$ and $P_3$ was vertex summed to $G_1$ at the center node of $P_3$.
    

\subsection{A graph with $k$ pendant vertices attached}
Let $G, \widehat G$ be as above. This section looks at the case where $G_2$ is a star graph $S_{k+1}$ and $v$ is the central vertex. This can also be viewed as attaching $k$ pendant vertices at a point $v\in G_1$. The main result in this section is that any set of edges added to $G_2$ in this set up will be a Braess set of edges so long as $G_1$ is sufficiently large. This generalizes work of \cite{ciardo2020braess}. 
We first find Kemeny's constant of a graph with $k$ pendant vertices attached. 

\begin{theorem}\label{thm:kemenykpendants}
Let $G = G_1\oplus_vS_{k+1}$ where $v$ is the central vertex of $S_{k+1}.$ Then
\[\kemeny(G) = \frac{m_1\kemeny(G_1) + k\mu(G_1,v) + k(m_1+k-\frac12)}{m_1+k}.\]
\end{theorem}
\begin{proof}
Using Proposition \ref{prop:somekemenys} and Theorem \ref{thm:kem1sepformula} we get
\begin{align*}
    \kemeny(G) &=\frac{m_1\left(\kemeny(G_1) + \mu(S_{k+1},v)\right)+k\left(\kemeny(S_{k+1}) + \mu(G_1,v)\right)}{m_1+k}\\
    &= \frac{m_1(\kemeny(G_1)+k) + k(k+1-\frac{3}{2}+\mu(G_1,v))}{m_1+k}\\
    &= \frac{m_1\kemeny(G_1)+k\mu(G_1,v) + k(m_1+k-\frac12)}{m_1+k}.
\end{align*}
\end{proof}

Now we state a result from electrical network theory that will be instrumental in the main proof of this section.
\begin{prop}[Mesh-Star Transform]\label{prop:meshstar}
A complete graph $K_n$ with unit resistance on each edge is equivalent to a star graph $S_{n+1}$ with resistance $\frac{1}{n}$ on each edge.
\end{prop}

\begin{theorem}\label{thm:starbraess}
Let $G=G_1\oplus_vS_{k+1}$ where $v$ is the central vertex of $S_{k+1}.$ Then any set of $l$ edges added in $S_{k+1}$ is a Braess set given
\begin{align*}
    |E(G_1)| >&\begin{cases}\frac{1}{8}\left[\sqrt{33l^2+50l+17}-l-1\right] &\text{if $l < k$}\\
    \frac{1}{8}\left[\sqrt{33k^2-30k+1}-k-1\right] &\text{if $l\geq k$}\end{cases}.
\end{align*}
\end{theorem}
\begin{proof}
Using Corollary \ref{cor:braessset} we can determine whether or not  $\mathscr{L}$, the $l$ added edges, is a Braess set. For convenience let $f$ denote the expression in Corollary \ref{cor:braessset}. Note that $m_2 = k$ so we have $f = Am^2_1+((A+C)k+Bl)m_1+C(k^2+lk)$. As can be seen by Corollary \ref{cor:braessset}, if $A>0$ then $f>0$ so long as $m_1$ is sufficiently large. For convenience, let $S_{k+1} = G_2.$

To be sure that $\mathscr{L}$ is a Braess set we will find lower bounds for the terms $A,B,C$. 
By Proposition \ref{prop:somekemenys} we know that $\mu(G_2,v) = k$ and $\kemeny(G_2) = k-\frac{1}{2}$.
Now we look at $A$.


To find a suitable lower bound on $\mu(\widehat{G}_2,v)$ we will find the smallest possible resistance distance from one of the $k$ pendants, $i$, to $v$ given we know the degree of $i$, $d_i$. By Rayleigh's monotonicity law (\cite{klein1993resistance}, Lemma D), adding an edge in a graph can only either decrease or have no effect on resistance distance between two points. It follows then that the resistance $\hat{r}_2(i,v)$ when $d_i = d$ is smallest when enough edges are added between $v$ and the $k-1$ pendant vertices that are not $i$ to make a complete graph $K_k$ and then adding edges connected to $i$ until $d_i = d$.

\begin{figure}[H]
    \centering
    \begin{tikzpicture}
    \tikzstyle{every node}=[circle, fill=black, inner sep=1pt]
    \foreach \x in {0, 72, 144, 216, 288}{
        \foreach \y in {0, 72, 144, 216, 288}{
            \ifthenelse{\equal{\x}{\y}}
            {}
            {\draw{(\x:1)node{}--(\y:1)node{}};}
        }
    }
    \draw{
    (0:2)node{}--(0:1)node{}
    (0:2)node{}--(72:1)node{}
    (0:2)node[label=$i$]{}--(288:1)node[label=below:{$v$}]{}
    };
\end{tikzpicture}
\qquad\qquad\qquad
\begin{tikzpicture}
    \tikzstyle{every node}=[circle, fill=black, inner sep=1pt]
    \node (1) at (0:0) {};
    \node (2) at (0:1) {};
    \node[label=$i$] (3) at (0:2) {};
    \node (4) at (72:1) {};
    \node (5) at (144:1) {};
    \node (6) at (216:1) {};
    \node[label=below:{$v$}] (7) at (288:1) {};
    
    \draw (3) -- (7) node [midway, minimum width = 2pt, fill=white] {\footnotesize $1$};
    \draw (3) -- (2) node [midway, fill=white] {\footnotesize $1$};
    \draw (3) -- (4) node [midway, fill=white] {\footnotesize $1$};
    \draw (1) -- (2) node [midway, fill=white] {\footnotesize $\frac{1}{5}$};
    \draw (1) -- (4) node [midway, fill=white] {\footnotesize $\frac{1}{5}$};
    \draw (1) -- (5) node [midway, fill=white] {\footnotesize $\frac{1}{5}$};
    \draw (1) -- (6) node [midway, fill=white] {\footnotesize $\frac{1}{5}$};
    \draw (1) -- (7) node [midway, fill=white] {\footnotesize $\frac{1}{5}$};
\end{tikzpicture}
    \caption{A mesh-star transformation of $K_5$ to $S_6$ with $d_i=3$}
    \label{fig:meshstar}
\end{figure}
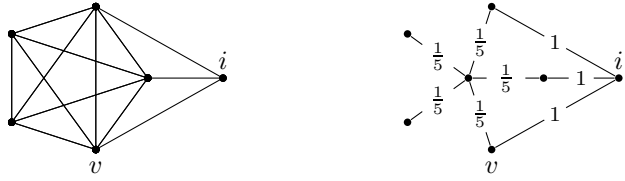

By Proposition \ref{prop:meshstar}, this is equivalent to a star graph $S_{k+1}$ with edge weights of $\frac{1}{k}$ and a vertex connected to $d$ of the pendants of the star by edge weights of $1$. So $\hat{r}_2(i,v)$ is the same as the resistance from the vertex of degree $d$ to one of the adjacent vertices. If $d=1$ then $\hat{r}_2(i,v)=1$. Suppose $d>1$. Then using the series and parallel rules for circuit reductions we have
\begin{align*}
    \hat{r}_2(i,v) =\frac{1}{\frac1 1+\frac{1}{\frac{1}{k}+\frac{1}{\frac{d-1}{1+\frac{1}{k}}}}}=\frac{1}{1+\frac{1}{\frac{1}{k}+\frac{1+\frac{1}{k}}{d-1}}}=\frac{1}{1+\frac{k(d-1)}{d+k}} =\frac{k+d}{d(k+1)}.
\end{align*}

So we have the following.
\begin{align*}
    \mu(\widehat{G}_2, v_2) = \hat d^T_2\widehat R_2e_{v_2} =&\sum_{i\in\widehat G_2\setminus{\{v\}}}\hat d_i\hat r_2(i,v)\\
    \geq&\sum_{i\in\widehat G_2\setminus{\{v\}}}\frac{k+d_i}{d_i(k+1)}\\
    =& \frac{1}{k+1}\left[\sum_{i\in\widehat G_2\setminus{\{v\}}}k + \sum_{i\in\widehat G_2\setminus{\{v\}}}d_i\right]\\
    =&\frac{k^2+k+2l}{k+1}
\end{align*}
 So we have $A = \mu(\widehat{G}_2,v) - \mu(G_2,v) \geq \frac{2l}{k+1}$. 

For the bounds on $B$ and $C$ we will make use of the fact proven in \cite{palacios2010bounds} that for an undirected graph on $n$ vertices, Kemeny's constant is smallest for a complete graph where $\kemeny(K_n) = \frac{(n-1)^2}{n}$. Then $B=\kemeny(\widehat G_2) - \mu(G_2,v) \geq \frac{-k}{k+1}$ and $C=\kemeny(\widehat G_2) - \kemeny(G_2) \geq\frac{-k}{k+1}+\frac{1}{2}$. So setting $B,C$ equal to their lower bounds we have $B = \frac{-k}{k+1}$ and $C=B + \frac1 2$.

Now let $\tilde f$ be $f$ where $A,B,C$ have been replaced by these lower bounds. Hence if $\tilde f>0$ then $f>0$. Then treating $\tilde f$ as a quadratic of variable $m_1$ it is seen that $\tilde f>0$ whenever
\begin{equation}\label{eq:pendantbound}
    |E(G_1)| = m_1 > \frac{1}{8l}\left[k^2-2kl-k+\sqrt{k(k^3-16l^2+2k^2(6l-1)+k(20l^2-12l+1))}\right]
\end{equation}

This bound on $m_1$ as is has some issues. For example, if $l=1$ it is known that the edge will be Braess for all nontrivial connected graphs. However, as $k$ increases (\ref{eq:pendantbound}) only guarantees this edge to be Braess for increasingly large graphs. For instance if $k=10$ and $l=1$ (\ref{eq:pendantbound}) would suggest we need $m_1 > 26$ which is unnecessarily large.

This behavior can be improved. For fixed $l$, (\ref{eq:pendantbound}) is monotonically increasing in $k$ for values of $k$ that make sense. Since often times adding $l$ edges won't affect all $k$ pendants we can look to treat such cases as having a smaller value of $k$. Since any edge added between twin pendant vertices will be Braess and since this function is increasing in $k$, we look for the largest number of pendants affected (new value of $k$ to work with) by adding an edge while never adding an edge between twin pendants if possible. For $l<k$ this new value of $k$ is $l+1$. Thus, if $l<k$ then plugging $k=l+1$ into the bound gives that a Braess set of edges is guaranteed if $m_1 > \frac{1}{8}\left[\sqrt{33l^2+50l+17}-l-1\right]$.

If $l \geq k$ then we can come to a cleaner (though slightly rougher) expression as well. The bound is strictly decreasing in $l$ for fixed $k$. Then for $l\geq k$, the substitution $l=k$ yields that the added edges form a Braess set whenever $m_1 > \frac{1}{8}\left[\sqrt{33k^2-30k+1}-k-1\right]$.

\end{proof}


\subsubsection{Pendant Triplets}\label{sec:pentrip}
While we now have a bound for conditions in which $l$ edges added to a group of $k$ pendants attached at a vertex in a graph are Braess, if we choose a specific $k$ we can work out more specifically how the addition of certain edges affects Kemeny's constant. Here we examine the case of $k=3$.

While these are the same types of graphs examined in Section \ref{sec:1sep}, to avoid excessive subscripts we redefine these specific graphs in a new way.

Let $G$ be a connected graph, $|V(G)|=n$, and $|E(G)| = m$. Let $\bar{G}$ be the graph obtained from $G$ by attaching 3 pendant vertices at $v$. We will first provide a formula for Kemeny's constant $\kemeny(\bar{G})$ in terms of the graph $G$.

\begin{lemma}\label{lem:kemG}
Let $\bar{G}$ be a graph as described above. 
Then
\begin{align*}
    \kemeny(\bar{G}) =\:\frac{2m\kemeny(G)+6\mu(G,v)+6m+15}{2m+6} 
\end{align*}
\end{lemma}
\begin{proof}
Using Theorem \ref{thm:kemenykpendants} with $k=3$ gives the result.
\end{proof}

Now we will define another graph $\widetilde{G}$ which is obtained from $\bar{G}$ by adding an edge between two of the vertices in $\{a,b,c\}$. For later comparisons we will derive an expression for the Kemeny's constant of $\widetilde{G}$ in terms of the graph $G$.

\begin{lemma}\label{lem:trip1edge}
Let $\widetilde{G}$ be as above. Then
\begin{align*}
    \kemeny(\widetilde{G}) =\: \frac{6m\kemeny(G)+24\mu(G,v)+22m+61}{6m+24} 
\end{align*}
\end{lemma}
\begin{proof}
It can be shown that in this case that $\kemeny(G_2) = \frac{61}{24}$ and $\mu(G_2, v) = \frac{11}{3}$. Using this, an easy simplification of Theorem \ref{thm:kem1sepformula} gives the result.
\end{proof}

Define the graph $\widehat{G}$ from $\widetilde{G}$ in the following way. Suppose $\widetilde{G}$ has an edge $ab$. Then adding the edge $bc$ will give $\widehat G$.

\begin{lemma}\label{lem:trip2edge}
Let $\widehat G$ be as above. Then
\begin{align*}
    \kemeny(\widehat G) =\: \frac{4m\kemeny(G) + 20\mu(G,v)+16m+47}{4m+20}
\end{align*}
\end{lemma}
\begin{proof}
It can be shown that $\kemeny(G_2) = \frac{47}{20}$ and $\mu(G_2, v) = 4$. The result follows from Theorem \ref{thm:kem1sepformula}.
\end{proof}

Define $G^*$ from $\widehat{G}$ by adding the edge $\{a, c\}$.
\begin{lemma}\label{lem:tripK4}
Let $G^*$ be as above. Then
\begin{align*}
    \kemeny(G^*) =\: \frac{2m\kemeny(G) + 12\mu(G, v) + 9m+27}{2m+12}
\end{align*}
\end{lemma}
\begin{proof}
In this case we have $G_2 = K_4$. By Proposition \ref{prop:somekemenys} $\kemeny(K_4) = \frac{9}{4}$ and $\mu(K_4,v) = \frac{9}{2}.$
\hidden{
The degree vector $d^*$ is given by the following.
\begin{align*}
    d^*_i = \begin{cases} d_i &\text{if $i\in G$}\\
    d_v+3 &\text{if $i = v$}\\
    3 & \text{if $i\in\{a,b,c\}$} \end{cases}
\end{align*}
The resistance matrix is given by the following.
\begin{align*}
    r^*_{ij} = \begin{cases} r_{ij} &\text{if $i,j\in G$}\\
    \frac{1}{2} &\text{if $i,j\in\{a,b,c\}, i\neq j$}\\
    r_{vj}+\frac{1}{2} &\text{if $i\in\{a,b,c\}, j\in G$} \end{cases}
\end{align*}
Then we have the following.
\begin{align*}
    d^{*T}R^* d^* =& \sum_{i,j\in G\setminus\{v\}}d^*_id^*_jr^*_{ij} + 2\sum_{j\in G\setminus\{v\}}d^*_vd^*_jr^*_{ij} + 2\sum_{\substack{i\in\{a,b,c\}\\j\in G\setminus\{v\}}}d^*_id^*_jr^*_{ij}\\
    &+ 2\sum_{j\in\{a,b,c\}}d^*_vd^*_jr^*_{vj} + \sum_{i,j\in\{a,b,c\}}d^*_id^*_jr^*_{ij}\\
    =&\: \sum_{i,j\in G\setminus\{v\}}d_id_jr_{ij} + 2\sum_{j\in G\setminus\{v\}}(d_v+3)d_jr_{vj} + 2\cdot3\sum_{j\in G\setminus\{v\}}3d_j(r_{vj}+\frac{1}{2})\\
    &+ 2\cdot3(d_v+3)\cdot3\cdot\frac{1}{2} + 6\cdot3\cdot3\cdot\frac{1}{2}\\
    =&\: d^TRd + 6\sum_{j\in G\setminus\{v\}}d_jr_{vj} + 18\sum_{j\in G\setminus\{v\}}d_jr_{vj} + 9\sum_{j\in G\setminus\{v\}}d_j+ 9d_v + 54\\
    =&\:d^TRd + 24d^TRe_v + 9(2m-d_v)+9d_v +54\\
    =&\: d^TRd+24d^TRe_v+18m+54
\end{align*}
Then by using Theorem \ref{thm:kemeny} and the fact that $|E(G^*)| = m+6$ yields the desired result.
}
\end{proof}

Now that the Kemeny's constant of each graph $\bar{G}, \widetilde{G}, \widehat G, G^*$ can be expressed in terms of the graph $G$, it is easy and interesting to compare these graphs to each other. The next five theorems will do just that.
\begin{theorem}
Let $G, \widetilde{G},$ and $\widehat G$ be as above. Suppose $|E(G)| \geq 4$. Then
\begin{align*}
    \kemeny(\widehat G) > \kemeny(\widetilde{G}).
\end{align*}
\end{theorem}
\begin{proof}
Using Lemmas \ref{lem:trip1edge}, \ref{lem:trip2edge} we have
\begin{align*}
    \kemeny(\widehat G) &> \kemeny(\widetilde{G})\\
    \frac{4m\kemeny(G)+20\mu(G,v)+16m+47}{4m+20} &>\frac{6m\kemeny(G)+24\mu(G,v)+22m+61}{6m+24}\\
    (4m\kemeny(G)+20\mu(G,v)+16m+47)(6m+24) &> (6m\kemeny(G)+24\mu(G,v)+22m+61)(4m+20)\\
    24m\mu(G,v)-24m\kemeny(G)+8m^2-18m-92 &> 0\\
    24m(\mu(G,v)-\kemeny(G))+2(4m^2-9m-46) &>0.
\end{align*}
The first term is proven to be non-negative in Proposition 3.2 of \cite{ciardo2020braess}. Treating the second term as a polynomial of variable $m$ it is seen that it is positive for integers at least 5. Checking all possibilities where $|E(G_1)| = 4$ one can find that this expression is still positive. However, it can be shown that this is negative for $G_1 = S_4$, thus it is not true for all graphs with $|E(G_1)| = 3$. 
\end{proof}

\begin{theorem}
Let $G, \widehat G,$ and $G^*$ be as above. Suppose $|E(G)| \geq2$. Then
\begin{align*}
    \kemeny(G^*) > \kemeny(\widehat G).
\end{align*}
\end{theorem}
\begin{proof}
Using Lemmas \ref{lem:trip2edge}, \ref{lem:tripK4} we have
\begin{align*}
    \kemeny(G^*) &> \kemeny(\widehat G)\\
    \frac{2m\kemeny(G)+12\mu(G,v)+9m+27}{2m+12} &> \frac{4m\kemeny(G)+20\mu(G,v)+16m+47}{4m+20}\\
    (2m\kemeny(G)+12\mu(G,v)+9m+27)(4m+20) &> (4m\kemeny(G)+20\mu(G,v)+16m+47)(2m+12)\\
    8m(\mu(G,v)-\kemeny(G))+4m^2+2m-24 &>0\\
    8m(\mu(G,v)-\kemeny(G))+2(2m^2+m-12) &> 0.
\end{align*}
Once again the first term is known to be non-negative by Proposition 3.2 of \cite{ciardo2020braess}. The second term as a polynomial of variable $m$ is positive for for integers at least 3. 

Now suppose $m=2$. Then $G = P_3$. Thus $\kemeny(G) = \frac{3}{2}$ and either $\mu(P_3,v) = 4$ or $\mu(P_3,v) = 2$. In both cases the above expression is positive, hence $\kemeny(G^*) > \kemeny(\hat G).$

Now suppose $m=1$. Then $G=P_2$, $\kemeny(G) = \frac{1}{2},$ and $\mu(G,v) = 1$. In this case, the expression is negative so $\kemeny(G^*) < \kemeny(\widehat G).$
\hidden{
Now suppose $m=2$. Then $d = \begin{bmatrix}1\\2\\1\end{bmatrix}$, and $R = \begin{bmatrix}0&1&2\\1&0&1\\2&1&0\end{bmatrix}$. There are two options for $e_v$. If $e_v = \begin{bmatrix}0\\0\\1\end{bmatrix}$ then the above expression is given by
\begin{align*}
    4md^TRe_v - d^TRd + (2m^2 + m -12) &= 8(4)-12-2 > 0.
\end{align*}
If $e_v = \begin{bmatrix}0\\1\\0\end{bmatrix}$ then we have
\begin{align*}
    4md^TRe_v - d^TRd + (2m^2 + m -12) &= 8(2)-12-2 > 0.
\end{align*}
Thus $\kemeny(G^*) > \kemeny(\hat G)$ when $m=2$.

Now suppose $m=1$. Then $d = \begin{bmatrix}1\\1\end{bmatrix}$, and $R = \begin{bmatrix}0&1\\1&0\end{bmatrix}$. This gives
\begin{align*}
    4md^TRe_v - \tau d^TRd + (2m^2 + m -12) &= 4 - 2 - 9 <0.
\end{align*}
Thus $\kemeny(G^*) < \kemeny(\hat G)$ when $m=1$.
}
\end{proof}

\begin{theorem}
Let $G, \widetilde{G},$ and $G^*$ be as above. Suppose $|E(G)| \geq4$. Then
\begin{align*}
    \kemeny(G^*) > \kemeny(\widetilde{G}).
\end{align*}
\end{theorem}
\begin{proof}
Using Lemmas \ref{lem:trip1edge}, \ref{lem:tripK4} we will show when the following is true.
\begin{align*}
    \kemeny(G^*) &> \kemeny(\widetilde{G})\\
    \frac{2m\kemeny(G)+12\mu(G,v)+9m+27}{2m+12} &> \frac{6m\kemeny(G)+24\mu(G,v)+22m+61}{6m+24}\\
    (2m\kemeny(G)+12\mu(G,v)+9m+27)(6m+24) &> (6m\kemeny(G)+24\mu(G,v)+22m+61)(2m+12)\\
    24m(\mu(G,v)-\kemeny(G))+2(5m^2-4m-42) &>0.
\end{align*}
Again, the first term is known to be non-negative by Proposition 3.2 of \cite{ciardo2020braess}. The second term as a polynomial of variable $m$ is positive for all integers at least 4. Thus the result is proved.
\end{proof}

\begin{theorem}
Let $G, \bar G,$ and $\widehat G$ be as above. Suppose $|E(G)| \geq 1$. Then
\begin{align*}
    \kemeny(\widehat G) > \kemeny(\bar G).
\end{align*}
\end{theorem}
\begin{proof}
Using Lemmas \ref{lem:kemG}, \ref{lem:trip2edge} we have
\begin{align*}
    \kemeny(\widehat G) &> \kemeny(\bar G)\\
    \frac{4m\kemeny(G)+20\mu(G,v)+16m+47}{4m+20} &> \frac{2m\kemeny(G)+6\mu(G,v)+6m+15}{2m+6}\\
    (4m\kemeny(G)+20\mu(G,v)+16m+47)(2m+6) &> (2m\kemeny(G)+6\mu(G,v)+6m+15)(4m+20)\\
    16m(\mu(G,v)-\kemeny(G))+2(4m^2+5m-9) &> 0.
\end{align*}
By Proposition 3.2 of \cite{ciardo2020braess} the first term is non-negative. Treating the second term as a polynomial of variable $m$ it is seen that it is 0 at $m=1$ and positive for all values of $m>1$. 

Now suppose $m=1$. Then $G=P_2$, $\kemeny(G) = \frac{1}{2},$ and $\mu(G,v) = 1$. Then the above expression is positive so $\kemeny(\widehat G) > \kemeny(\bar G).$

\end{proof}

\begin{theorem}
Let $G, \bar G,$ and $G^*$ be as above. Suppose $|E(G)| \geq 1$. Then
\begin{align*}
    \kemeny(G^*) \geq \kemeny(\bar G)
\end{align*}
with equality if and only if $G = P_2$.
\end{theorem}
\begin{proof}
Using Lemmas \ref{lem:kemG}, \ref{lem:tripK4} we have
\begin{align*}
    \kemeny(G^*) &\geq \kemeny(\bar G)\\
    \frac{2m\kemeny(G)+12\mu(G,v)+9m+27}{2m+12} &\geq \frac{2m\kemeny(G)+6\mu(G,v)+6m+15}{2m+6}\\
    (2m\kemeny(G)+12\mu(G,v)+9m+27)(2m+6) &\geq (2m\kemeny(G)+6\mu(G,v)+6m+15)(2m+12)\\
    12m(\mu(G,v)-\kemeny(G))+6(m^2+m-3) &\geq 0.
\end{align*}
By Proposition 3.2 of \cite{ciardo2020braess} the first term is non-negative. The second term as a polynomial of variable $m$ is positive for all integers at least 2.

Now suppose $m=1$. Then $G=P_2$, $\kemeny(G) = \frac{1}{2},$ and $\mu(G,v) = 1$. Then the above inequality is an equality.

\end{proof}
In summary in the $k=3$ case for graphs of this construction we see that for almost all graphs, the more edges get added in the pendant portion of the graph, the larger Kemeny's constant will be. One of the few graphs where this is not the case is $G = P_2.$ In this case we have that $\kemeny(\bar G) = \kemeny(G^*) < \kemeny(\widehat G) < \kemeny(\widetilde G)$, so each additional edge actually decreases the Kemeny constant slightly until it is back to where it started.

\bibliographystyle{plain}
\bibliography{sources}

\begin{thebibliography}{10}

\bibitem{bapat2010graphs}
Ravindra~B Bapat.
\newblock {\em Graphs and matrices}, volume~27.
\newblock Springer, 2010.

\bibitem{barrett2019resistance}
Wayne Barrett, Emily~J Evans, and Amanda~E Francis.
\newblock Resistance distance in straight linear 2-trees.
\newblock {\em Discrete Applied Mathematics}, 258:13--34, 2019.

\bibitem{braess1968paradoxon}
Dietrich Braess.
\newblock {\"U}ber ein paradoxon aus der verkehrsplanung.
\newblock {\em Unternehmensforschung}, 12(1):258--268, 1968.

\bibitem{braess2005paradox}
Dietrich Braess, Anna Nagurney, and Tina Wakolbinger.
\newblock On a paradox of traffic planning.
\newblock {\em Transportation science}, 39(4):446--450, 2005.

\bibitem{breen2019computing}
Jane Breen, Steve Butler, Nicklas Day, Colt DeArmond, Kate Lorenzen, Haoyang
  Qian, and Jacob Riesen.
\newblock Computing {K}emeny's constant for a barbell graph.
\newblock {\em The Electronic Journal of Linear Algebra}, 35:583--598, 2019.

\bibitem{catral2010kemeny}
Minerva Catral, Stephen~J Kirkland, Michael Neumann, and N-S Sze.
\newblock The {K}emeny constant for finite homogeneous ergodic markov chains.
\newblock {\em Journal of Scientific Computing}, 45(1):151--166, 2010.

\bibitem{ciardo2020braess}
Lorenzo Ciardo.
\newblock The {B}raess' {P}aradox for {P}endant {T}wins.
\newblock {\em Linear Algebra and its Applications}, 2020.

\bibitem{ciardo2020kemeny}
Lorenzo Ciardo, Geir Dahl, and Steve Kirkland.
\newblock On {K}emeny's constant for trees with fixed order and diameter.
\newblock {\em Linear and Multilinear Algebra}, pages 1--23, 2020.

\bibitem{kirkland2016kemeny}
Steve Kirkland and Ze~Zeng.
\newblock Kemeny's constant and an analogue of {B}raess' paradox for trees.
\newblock {\em Electronic Journal of Linear Algebra}, 31(1):444--464, 2016.

\bibitem{klein1993resistance}
Douglas~J Klein and Milan Randi{\'c}.
\newblock Resistance distance.
\newblock {\em Journal of Mathematical Chemistry}, 12(1):81--95, 1993.

\bibitem{levene2002kemeny}
Mark Levene and George Loizou.
\newblock Kemeny's constant and the random surfer.
\newblock {\em The American Mathematical Monthly}, 109(8):741--745, 2002.

\bibitem{li2019multiplicative}
Shuchao Li, Wanting Sun, and Shujing Wang.
\newblock Multiplicative degree-kirchhoff index and number of spanning trees of
  a zigzag polyhex nanotube tuhc [2n, 2].
\newblock {\em International Journal of Quantum Chemistry}, 119(17):e25969,
  2019.

\bibitem{palacios2011broder}
Jos{\'e}~Luis Palacios and Jos{\'e}~M Renom.
\newblock Broder and {K}arlin's formula for hitting times and the {K}irchhoff
  index.
\newblock {\em International Journal of Quantum Chemistry}, 111(1):35--39,
  2011.

\bibitem{palacios2010bounds}
Jos{\'e}~Luis Palacios and Jos{\'e}~Miguel Renom.
\newblock Bounds for the {K}irchhoff index of regular graphs via the spectra of
  their random walks.
\newblock {\em International Journal of Quantum Chemistry}, 110(9):1637--1641,
  2010.

\bibitem{patel2015robotic}
Rushabh Patel, Pushkarini Agharkar, and Francesco Bullo.
\newblock Robotic surveillance and {M}arkov chains with minimal weighted
  {K}emeny constant.
\newblock {\em IEEE Transactions on Automatic Control}, 60(12):3156--3167,
  2015.

\end{thebibliography}

\end{document}